\documentclass[10pt]{amsart}

\IfFileExists{srcltx.sty}{\usepackage{srcltx}}

\numberwithin{equation}{section}

\usepackage[english,french]{babel}

\usepackage[latin1]{inputenc}
\usepackage{xspace,amssymb,amsfonts,euscript}
\usepackage{amsthm,amsmath}
\usepackage{palatino}
\usepackage{euscript}
\input xy \xyoption {all}

\RequirePackage{color}
\definecolor{myred}{rgb}{0.75,0,0}
\definecolor{mygreen}{rgb}{0,0.5,0}
\definecolor{myblue}{rgb}{0,0,0.65}

\RequirePackage{ifpdf}
\ifpdf
  \IfFileExists{pdfsync.sty}{\RequirePackage{pdfsync}}{}
  \RequirePackage[pdftex,
   colorlinks = true,
   urlcolor = myblue, 
   citecolor = mygreen, 
   linkcolor = myred, 
   pagebackref,
   bookmarksopen=true]{hyperref}
\else
  \RequirePackage[hypertex]{hyperref}
\fi

\RequirePackage{ae, aecompl, aeguill} 

    \def\CM{{\mathbb{C}}}
    \def\DM{{\mathbb{D}}}

  \def\gg{{\mathfrak g}}  \def\GM{{\mathbb{G}}}

    \def\KM{{\mathbb{K}}}

    \def\NM{{\mathbb{N}}}
    \def\OM{{\mathbb{O}}}
    \def\PM{{\mathbb{P}}}
    \def\QM{{\mathbb{Q}}}
    \def\RM{{\mathbb{R}}}

    \def\ZM{{\mathbb{Z}}}

\def\BB{{\mathbf B}}    
    \def\CC{{\mathcal{C}}}
    \def\DC{{\mathcal{D}}}
    \def\EC{{\mathcal{E}}}
    \def\FC{{\mathcal{F}}}
\def\GB{{\mathbf G}}    \def\GC{{\mathcal{G}}}
\def\HB{{\mathbf H}}    \def\HC{{\mathcal{H}}}
    
    \def\JC{{\mathcal{J}}}
    \def\KC{{\mathcal{K}}}
\def\LB{{\mathbf L}}    \def\LC{{\mathcal{L}}}

    \def\OC{{\mathcal{O}}}
    \def\PC{{\mathcal{P}}}

\def\TB{{\mathbf T}}    \def\TC{{\mathcal{T}}}

\def\XB{{\mathbf X}}    
\def\YB{{\mathbf Y}}

\def\ll{\lambda}
\def\LL{\Lambda}

\newcommand{\nc}{\newcommand} \newcommand{\renc}{\renewcommand}

\newcommand{\rdots}{\mathinner{ \mkern1mu\raise1pt\hbox{.}
    \mkern2mu\raise4pt\hbox{.}
    \mkern2mu\raise7pt\vbox{\kern7pt\hbox{.}}\mkern1mu}}

\DeclareMathOperator{\Res}{Res}

\def\ov{\overline}

\def\p{{}^p}

\def\to{\rightarrow}

\def\longto{\longrightarrow}

\nc{\triright}{\stackrel{[1]}{\to}}
\nc{\longtriright}{\stackrel{[1]}{\longto}}

\nc{\Br}{\mathcal{B}}
\nc{\HotRR}{{}_R\mathcal{K}_R}
\nc{\HotR}{\mathcal{K}_R}
\nc{\excise}[1]{}
\nc{\defect}{\text{df}}
\nc{\h}[1]{\underline{H}_{#1}}

\nc{\Ga}{\mathbb{G}_a} 
\nc{\Gm}{\mathbb{G}_m} 

\nc{\Perv}{{\mathbf{P}}}

\nc{\IH}{{\mathrm{IH}}}

\nc{\ic}{\mathbf{IC}}

\nc{\gl}{{\mathfrak{gl}}}
\renc{\sl}{{\mathfrak{sl}}}
\renc{\sp}{{\mathfrak{sp}}}

\newcommand{\Sp}{{\mathbf{Sp}}}

\nc{\HBM}{H^{BM}}

 \DeclareMathOperator{\Hom}{Hom}
\DeclareMathOperator{\supp}{supp} 
 \DeclareMathOperator{\Rep}{Rep}

\DeclareMathOperator{\Bs}{BS}

\newtheorem{thm}{Theorem}[section]
\newtheorem{lem}[thm]{Lemma}

\newtheorem{prop}[thm]{Proposition}
\newtheorem{cor}[thm]{Corollary}
  
\newtheorem{conj}[thm]{Conjecture}

\theoremstyle{definition}
\newtheorem{defi}[thm]{Definition}

\theoremstyle{remark}
\newtheorem{remark}[thm]{Remark}

\DeclareMathOperator{\Ext}{Ext}

\def\uk{\underline{k}}

\def\Gr{{\EuScript Gr}}

\def\GD{\check{G}}

\def\TD{\check{T}}
\def\LD{\check{L}}
\def\HD{\check{H}}

\def\hsr{{\alpha_0}}

\begin{document}

\title[Parity sheaves and tilting modules]{
Parity sheaves and tilting modules}

\author{Daniel Juteau} \address{ LMNO, Universit\'e de Caen
  Basse-Normandie, CNRS, BP 5186, 14032 Caen, France}
\email{daniel.juteau@unicaen.fr}
\urladdr{http://www.math.unicaen.fr/~juteau}

\author{Carl Mautner} \address{ Max Planck Institut f\"ur Mathematik,
Vivatsgasse 7, 53111 Bonn, Germany}
\email{cmautner@mpim-bonn.mpg.de}
\urladdr{http://people.mpim-bonn.mpg.de/cmautner}

\author{Geordie Williamson} \address{ Max Planck Institut f\"ur Mathematik,
Vivatsgasse 7, 53111 Bonn, Germany}
\email{geordie@mpim-bonn.mpg.de}
\urladdr{http://people.mpim-bonn.mpg.de/geordie}

\thanks{D.J. was supported by ANR grant ANR-09-JCJC-0102-01 and
  C.M. by an NSF postdoctoral fellowship.}

\maketitle

\selectlanguage{english}
\begin{abstract}
We show that tilting modules and parity
sheaves on the affine Grassmannian are related through the geometric
Satake correspondence, when the characteristic is bigger than an
explicit bound.
\end{abstract}


\selectlanguage{english}


\section{Introduction}


\subsection{Tilting modules for reductive groups} 
\label{subsec-redgps}
Let $\GB$ be a split reductive group over a field $k$ of
characteristic $p$ with a chosen maximal torus and Borel subgroup $\TB
\subset \BB \subset \GB$.  Let $\LL^+$ denote the set of dominant
weights.  To each dominant weight $\lambda \in \Lambda^+$, is
associated an induced representation $\nabla_\lambda:= \mathrm{ind}_\BB^\GB k_\lambda =
H^0(\GB/\BB,\OC(\lambda))$ and its dual $\Delta_\lambda$, the Weyl
module.

The rational representations of $\GB$ form a highest weight category in
which the Weyl modules are the standard objects and the induced modules are
the costandard objects.  A rational representation is said to be
tilting, if it admits two filtrations --- one with successive quotients
isomorphic to Weyl modules and the other with successive quotients
isomorphic to induced modules.

A theorem of Ringel~\cite[Proposition 2]{Ringel} about general highest
weight categories specializes in this setting to the following
result~\cite[Theorem 1.1]{Donkin},

\begin{thm}
\label{thm-RingelTilt}
For each $\ll \in \LL^+$, (up to non-canonical isomorphism) there
exists a unique indecomposable tilting module $T(\ll)$ which has a
unique highest weight $\ll$.  Moreover, $\ll$ has multiplicity one as
a weight of $T(\ll)$.  Every indecomposable tilting module is
isomorphic to $T(\ll)$ for some $\ll \in \LL^+$.
\end{thm}

An interesting feature of the class of tilting modules is that it is
closed under both tensor product and restriction to a Levi subgroup:

\begin{thm}
\label{thm-tensor}
If $T$ and $T'$ are tilting modules for $\GB$, then so is the tensor
product $T \otimes T'$.
\end{thm}

\begin{thm}
\label{thm-res}
Let $\LB$ be a Levi subgroup of $\GB$.  If $T$ is a tilting module for
$\GB$, then the restriction $\Res^\GB_\LB T$ to $\LB$ is a tilting
module for $\LB$.
\end{thm}

The first of these theorems was originally proven by Wang~\cite{WangFilt}
in type $A$ and in large characteristic for other groups.  Donkin~\cite{Donkin:book}
later proved both theorems in almost full generality (he excluded the
case when $p=2$ and $G$ has a component of type $E_7$ or $E_8$).
The first complete and uniform proof of both theorems is due to
Mathieu~\cite{Mathieu} and uses Frobenius splitting techniques.  Other approaches to
the first theorem appear in~\cite{Littelmann, PoloFilt, Paradowski,
  Kaneda}.\footnote{In the literature, these theorems appear with the
  words `tilting modules' replaced by `modules admitting a good
  filtration' or `modules admitting a Weyl filtration'.  In the
  appendix to this paper, we explain the fact, well-known to experts, that these are equivalent
  formulations.}

\subsection{Parity sheaves for the affine Grassmannian}
\label{subsec-parity}

Let $\GD \supset \TD$ be the connected complex algebraic group and
maximal torus with root datum dual to that of $\GB$.

Let $\KC = \CM((t))$ and $\OC= \CM[[t]]$.
The affine Grassmannian $\Gr$ for $\GD$ is an ind-scheme whose complex
points form the set $\GD(\KC)/\GD(\OC)$.  We consider its
complex points as an ind-$\GD(\OC)$-variety. The $\GD(\OC)$-orbits are
labeled by the set $\Lambda^+$ of dominant  weights of $\GB$ and we
denote the orbit corresponding to a weight $\lambda$ by $\Gr^\lambda$.

The geometric Satake theorem~\cite{MV} shows that the representation
theory of $\GB$ is encoded in a category of sheaves of $k$-vector
spaces on $\Gr$. More precisely, the category of rational
representations of $\GB$ is equivalent to the category $P(\Gr)$: the
category of $\GD(\OC)$-equivariant perverse sheaves on $\Gr$ with
coefficients in $k$.

The category $P(\Gr)$ is the heart of a t-structure on $D(\Gr)$, the bounded
$\GD(\OC)$-equivariant constructible derived category of sheaves of $k$-vector
spaces on $\Gr$.  There is a natural convolution product $\star:
D(\Gr) \times D(\Gr) \to D(\Gr)$ which is t-exact and produces a
tensor structure on $P(\Gr)$, corresponding under the equivalence
to the tensor product of rational representations of $\GB$.

Similarly, for any Levi subgroup $\LB \subset \GB$, the restriction
functor from $\GB$ to $\LB$ corresponds to a geometrically-defined
t-exact functor $R^{\GD}_{\LD}: D(\Gr) \to D(\Gr_{\LD})$, where
$\Gr_{\LD}$ is the affine Grassmannian for $\LD \subset \GD$, the Levi
subgroup containing $\TD$ whose roots are dual to those of $\LB$ (see
Section~\ref{subsec-hyperlocGr} for more details).

\medskip

Recall the notion of a parity complex \cite[Section
2.2]{JMW2}.\footnote{Unless stated otherwise, in this paper parity
  complexes are defined with respect to the constant pariversity
  $\natural$.}
The affine Grassmannian is a Kac-Moody flag variety
and hence the results from ~\cite[Section 4.1]{JMW2} (see in
particular Example 4.2 of loc. cit.) can be used to study
$D(\Gr)$. Parity complexes on the affine Grassmannian behave very much
like the tilting modules for $\GB$.  In particular, we have the
following theorems which mirror the ones for tilting modules.

The starting point is a result~\cite[Theorem 4.6]{JMW2} that is very
similar to Theorem~\ref{thm-RingelTilt}:

\begin{thm}
Assume that the characteristic of $k$ is not a torsion prime for
$\GD$.\footnote{This restriction can be removed by working in the
  non-equivariant setting.}

For each $\ll \in \LL^+$, (up to non-canonical isomorphism) there
exists a unique indecomposable parity complex $\EC(\ll)$ such that
$\supp(\EC(\ll)) = \ov{\Gr^\ll}$ and $\EC(\ll)|_{\Gr^\ll} =
\uk_{\Gr^\ll}[\dim \Gr^\ll]$.  Every indecomposable parity complex is
isomorphic to $\EC(\ll)$ for some $\ll\in \LL^+$.
\end{thm}

The indecomposable parity complexes $\EC(\ll)$ are known as parity sheaves.

As a special case of \cite[Theorem 4.8]{JMW2}, we obtain an analogue of Theorem~\ref{thm-tensor}:

\begin{thm}
\label{thm-parityconv}
If $\FC \in D(\Gr)$ and $\GC \in D(\Gr)$ are parity complexes, then so
is the convolution product $\FC \star \GC \in D(\Gr)$.
\end{thm}

The first part of this paper establishes an analogue of Theorem~\ref{thm-res}.
In Section \ref{sec-hyperloc}, we prove the following result,

\begin{thm}
\label{thm-parityres}
Let $\LD$ be the Langlands dual of $\LB$ a Levi subgroup of $\GB$. If
$\FC \in D(\Gr)$ is a parity complex, then $R^{\GD}_{\LD}(\FC)$ is a
parity complex on the affine Grassmannian for $\LD$.
\end{thm}

The idea of the proof is to replace the purity argument
of~\cite[Theorem 2]{Br} by a parity argument.

\subsection{Tilting equals parity}
\label{subsec-geomsat}

In Section~\ref{sec-tilting}, which can be read independently of
Section~\ref{sec-hyperloc}, we prove our main result, which explains the similarities between
the theorems stated above.  Our result shows that, for most characteristics, the
Theorems~\ref{thm-tensor} and~\ref{thm-res} about tilting modules are
equivalent to the Theorems~\ref{thm-parityconv} and \ref{thm-parityres} about parity sheaves.

Recall~\cite[Prop. 13.1]{MV} that, for $\lambda \in \Lambda^+$ a dominant weight,
the Weyl module $\Delta_\lambda$ (resp. $\nabla_\lambda$) goes under the
geometric Satake equivalence to the standard sheaf 
${}^p\JC_{!}(\lambda):={}^p j_{\lambda !} \uk_{\lambda}[d_\lambda]$
(resp. costandard sheaf ${}^p\JC_{*}(\lambda):={}^p j_{\lambda *} \uk_{\lambda}[d_\lambda]$)
where $j_\lambda:\Gr^\lambda \to \Gr$ denotes the inclusion,
 $\uk_\lambda$ the constant sheaf on $\Gr^\lambda$ and $d_\lambda$
 the dimension of $\Gr^\lambda$.

We say $\FC \in P(\Gr)$ is a tilting sheaf if it corresponds to a tilting
module for $\GB$.  This is equivalent to admitting two filtrations ---
one with standard successive quotients and the other with costandard
successive quotients.\footnote{Warning: this definition of tilting sheaf is more 
  general than that of~\cite{TiltExer}, which does not apply to
  this setting.}
We denote by $\TC(\ll)$ the tilting sheaf corresponding to the
indecomposable tilting module $T(\lambda)$.

Our main theorem is the following geometric characterization of the tilting sheaves on the affine Grassmannian.
%
%
We will need to assume that the characteristic $p$ is bigger than some bound depending only on the root system $\Phi$ of $\GB$.

\begin{defi}
If the root system $\Phi$ is irreducible, let $b(\Phi)$ be given by the following table:
\begin{table}[h]
\[
\begin{array}{|c||c|c|c|c|c|c|}
\hline
\text{Type of } \Phi & A_n & B_n, D_n & C_n & G_2, F_4, E_6 & E_7 & E_8\\
\hline
b(\Phi) & 1 &  2 &  n &  3 & 19 & 31\\
\hline
\end{array}
\]
\caption{Table of bounds}\label{table:primes}
\end{table}

In the general case, let $\Phi = \sqcup_{i=1}^{s} \Phi_{i}$ be the
decomposition into irreducible components. Then we set
$b(\Phi) := \max_{1\leq i\leq s} b(\Phi_{i})$.
\end{defi}

\begin{thm}
\label{thm:T=P}
If $p > b(\Phi)$, then the group $\GB$ satisfies
\begin{equation}\tag{$*$}
\forall \lambda \in \Lambda^+, \quad
\EC(\lambda) = \TC(\lambda).
\end{equation}
\end{thm}

In particular, for $p$ as in the Theorem, every $\EC(\lambda)$ is
perverse. Note that we know examples where $\EC(\lambda)$ fails to be
perverse, for bad primes (see Lemma \ref{lem-MinuHSR}). However,
Lemma \ref{lem-MinuHSR}(4) suggests that the following may always be
true:

\begin{conj}
In arbitrary characteristic,
for every dominant weight $\lambda \in \Lambda^+$,
the perverse sheaf $\p\HC^0 \EC(\lambda)$ is tilting.
\end{conj}

On the other hand, note that by Proposition \ref{prop-paritytilting},
the property $(*)$ is actually equivalent to all $\EC(\lambda)$ being
perverse.

\subsection{Applications}  

One motivation for this work was a conjecture of Mirkovi\'c
and Vilonen~\cite[Conjecture 13.3]{MV}. They conjectured that the
stalks of standard sheaves with $\ZM$-coefficients on the affine
Grassmannian are torsion free. This conjecture is equivalent
to the standard sheaves being $*$-parity for all fields (i.e. their non-zero stalks should be concentrated in one parity). Actually the
minimal nilpotent orbit singularities provide counterexamples to this
conjecture, in all types but in type $A_n$: see \cite{Ju-AffGr}, where
the conjecture is modified to exclude bad primes. We get the following
reformulation: 

\begin{conj}
\label{conj-mv}
 If $p$ is a good prime for $\GB$, then the standard sheaves
  with coefficients in a field of characteristic $p$ are $*$-parity.
\end{conj}

Note that if Conjecture~\ref{conj-mv} is true, then it implies that
$\GB$ satisfies $(*)$ whenever $p$ is a good prime.

Conversely, since an earlier draft of the current paper was circulated,
Achar-Rider~\cite{ARi} proved that if $\GB$ satisfies $(*)$, then the
Mirkovi\'{c}-Vilonen conjecture is correct.  In particular, using
our Theorem~\ref{thm:T=P}, they settle the conjecture for all but a
handful of cases.

\medskip

Our results also may be used to obtain new proofs of
Theorems~\ref{thm-tensor} and~\ref{thm-res} in most characteristics.
To see this note that for all groups for which the hypothesis of
Theorem~\ref{thm:T=P} applies, our Theorems~\ref{thm-parityconv}
and~\ref{thm-parityres} imply Theorems~\ref{thm-tensor}
and~\ref{thm-res}.  The careful reader will observe that in proving
Theorem~\ref{thm:T=P} we do not use any results that rely on
Theorem~\ref{thm-tensor} or~\ref{thm-res}.

\medskip

Lastly, in Section~\ref{subsec-qchar}, we observe that our main result implies 
the existence (for most characteristics) of $q$-characters for tilting
modules, meaning a natural $q$-analogue of the characters of tilting
modules.

\subsection{Acknowledgments}
We would like to thank Steve Donkin and Jens Carsten Jantzen for substantial
assistance with tilting modules.  The second and third authors are
also very grateful to the Max Planck Institute for wonderful working conditions.

\section{Hyperbolic Localization}
\label{sec-hyperloc}

The goal of this section is to prove Theorem~\ref{thm-parityres}.
In~\ref{subsec:setup}, we recall the notion of hyperbolic localization
and Braden's Theorem.
In~\ref{subsec-evenstalks}, we introduce some simplifying assumptions
and study the hyperbolic localization of parity sheaves in this setting.
In~\ref{subsec-hyperlocGr}, we apply the result
of~\ref{subsec-evenstalks} to the affine Grassmannian and prove
Theorem~\ref{thm-parityres}.

\subsection{Braden's Theorem} \label{subsec:setup}
Let $T$ be a complex torus and $X$ a complex $T$-variety.

We make the following assumption, which is automatic if $X$ is normal by Sumihiro's theorem \cite{Sumihiro, KKLV}:
\begin{equation}\tag{C}\label{affine}
\text{$X$ has a covering by $T$-stable affine open
subvarieties.}
\end{equation}


Let $\chi : \GM_m \to T$ be a cocharacter of $T$.  Our goal is to
understand the hyperbolic localization of $\pi_* \uk_Y$ with respect
to the $\GM_m$-action defined by $\chi$.  We begin by recalling
Braden's definition.

Let $Z \subset X$ denote the variety of $\chi$-fixed points and $Z_1,\ldots,Z_m$ its connected components.
Consider the attracting and repelling varieties for each component $Z_i$, $1 \leq i \leq m$:
\[ Z^+_i = \{ x \in X | \lim_{s \to 0} \chi(s) \cdot x \in Z_i \}, \]
\[ Z^-_i = \{ x \in X | \lim_{s \to \infty} \chi(s) \cdot x \in Z_i \}. \]

Let $Z^+$ (respectively $Z^-$) be the disjoint, disconnected union of the $Z^+_i$ (respectively $Z^-_i$) and the maps $f^\pm : Z \to Z^\pm$ and $g^{\pm}: Z^\pm \to X$ be the component-wise inclusions.  Define projection maps $p^{\pm}: Z^\pm \to Z$ by 
$p^+(x) = \lim_{t \to 0} \chi(t) \cdot x$ and $p^-(x) = \lim_{t \to \infty} \chi(t) \cdot x$.
These are algebraic maps by \cite[Proposition 4.2]{Hesselink}.\footnote{In
this article we only consider attractive sets, which corresponds to
the speed $m = 1$ case in the setting of \cite{Hesselink}. One can
reduce to the affine case using \cite[Lemma 4.4]{Hesselink} because of our
standing assumption that $X$ admits a covering by $T$-stable open
affine subvarieties.}

For this section, we let $D(X)$ denote the constructible derived
category of sheaves of $k$-vector spaces on $X$.

Recall that the hyperbolic localization functors
$(-)^{!*},(-)^{*!}:D(X) \to D(Z)$ for the character $\chi$ are defined by
\[ \FC^{!*} := (f^+)^!(g^+)^* \FC,\]
\[ \FC^{*!} := (f^-)^*(g^-)^! \FC.\]


We will use the following results of Braden:

\begin{thm}[\cite{Br}, Theorem 1]
\label{thm:Braden}
For any $\FC \in D(X)$, there is a natural morphism $\iota_\FC : (\FC)^{*!} \to (\FC)^{!*}$.
If $\FC$ is weakly equivariant (e.g., comes from an object in the equivariant derived category), then
\begin{enumerate}
\item[(i)]
there are natural isomorphisms $\FC^{!*} \cong (p^+)_! (g^+)^* \FC$ and $\FC^{*!} \cong (p^-)_* (g^-)^! \FC$ and
\item[(ii)]
the morphism $\iota_\FC : \FC^{*!} \to \FC^{!*}$ is an isomorphism.
\end{enumerate}
\end{thm}

Using this result, Braden proves that for $k=\QM$, hyperbolic
localization of the intersection cohomology complex $\ic(X;\QM)$ is a direct sum of
shifted intersection cohomology complexes. Our goal here is to prove a
similar result for certain parity complexes.

\subsection{Parity of stalks at $T$-fixed points}
\label{subsec-evenstalks}

We keep the notation of \S \ref{subsec:setup}.

Let $Y$ be a smooth projective $T$-variety and $\pi:Y \to X$ a
$T$-equivariant proper morphism.
Furthermore, we will assume that:
\begin{enumerate}
\item The sets of $T$-fixed points, $X^T$ and $Y^T$, are finite.
\item For any $T$-fixed point $x \in X^T$, there exists a cocharacter $\GM_m \to T$ such that $\GM_m$ acts attractively on a neighborhood of $x$.
\item There exists a $T$-module $V$, for which $Y$ admits a closed $T$-equivariant embedding into the projective space $\PM(V)$.
\end{enumerate}

\begin{prop} \label{prop:evenstalks}
The cohomology of the stalk of $(\pi_* \uk_Y)^{!*}$ at any $z \in X^T$ is concentrated in even degrees.
\end{prop}

\begin{proof}
The push-forward $\pi_* \uk_Y$ is weakly equivariant and thus by (i) of Theorem \ref{thm:Braden}, $(\pi_* \uk_Y)^{!*} \cong (p^+)_! (g^+)^* \pi_* \uk_Y$. 

Let $C^+$ be defined by the Cartesian square
\[
\xymatrix{
C^+ \ar[r] \ar[d] & Y \ar[d] \\
Z^+ \ar[r] & X .
}
\]
In other words, it is the disjoint union of the preimages $\pi^{-1}(Z_i^+)$, $1
\leq i \leq m$. We abuse notation and also write $\pi$ for the induced
map $\pi : C^+ \to Z^+$. The map $\pi$ is proper and so by base change
we have 
\[ (\pi_* \uk_Y)^{!*} = (p^+)_!  \pi_! \uk_{C^+}. \]

The stalk of $(\pi_* \uk_Y)^{!*}$ at $z$ thus has cohomology $H^*_c ( (p^+ \circ
  \pi)^{-1}(z))$. 
Using the following lemma, we will construct a Bia{\l}ynicki-Birula
decomposition on $Y$ which restricts to one on $(p^+ \circ
\pi)^{-1}(z)$.

\begin{lem}
For any $z \in X^T$, consider the direct sum decomposition of the
Zariski tangent space $T_z X \cong V^+ \oplus V^0 \oplus V^-$, such that
$\chi$ acts on $V^+$, $V^0$, and $V^-$ with positive,
zero, and negative weights respectively.  There exists a cocharacter $\zeta:
\GM_m \to T$ whose action on $T_z X$ is attractive
on $V^+$ and repulsive on $V^0 \oplus V^-$, and such that $Y^\zeta = Y^T$.
\end{lem}

\begin{proof}
Let $\YB(T)$ denote the cocharacters of $T$.
The set of cocharacters which act on $T_z X$ with negative weights is
the intersection of $\YB(T)$ with an open cone in $\YB(T)\otimes \RM$.
By assumption \ref{subsec-evenstalks}(2), this intersection is nonempty.
In other words, there exists a $\rho \in \YB(T)$ that acts on $T_z X$
with negative weights.

At each fixed point $y\in Y^T$, the action of $T$ on $T_y Y$ splits as
a direct sum of characters. By assumption \ref{subsec-evenstalks}(1),
the set $\{\beta_i\}$ of all characters of $T$ obtained in this way is
finite.  As $Y$ is smooth, for any cocharacter $\sigma \in \YB(T)$,
$Y^\sigma = Y^T$ if and only if $\sigma$ is not contained in one of
the (finitely many) hyperplanes $\langle \sigma,\beta_i \rangle = 0$.
Thus we can choose $\rho$ above such that it acts on $T_z X$ with
negative weights and $Y^\rho = Y^T$.

Let $\{\alpha_k^+\}$ denote the set of characters of $T$ occurring in $V^+$.
Thus $\langle \chi,\alpha^+_k \rangle$ is positive and $\langle
\rho,\alpha^+_k \rangle$ is negative.  For $m \in \ZM$ sufficiently
large, $\langle m \chi + \rho, \alpha^+_k \rangle = m \langle
\chi,\alpha^+_k \rangle + \langle \rho,\alpha^+_k \rangle$ is positive
for all $k$.  Thus for $m$ large enough the cocharacter $m \chi +
\rho$ acts attractively on $V^+$.  On the other hand, for any
$m > 0$, $m\chi + \rho$ has strictly negative weights on $V^0 \oplus
V^-$. 

Lastly, as $\rho$ has been chosen such that $\langle \rho, \beta_i
\rangle \neq 0$ for all $i$,  $\langle m \chi + \rho,\beta_i\rangle$
will also be non-zero for $m$ sufficiently large.

For such an $m$, we may define $\zeta$ to be $m\chi + \rho$.
\end{proof}

Consider the attracting Bia{\l}ynicki-Birula decomposition of $Y$ with
respect to $\zeta$.  By the assumption \eqref{affine} in Section \ref{subsec:setup}, a $T$-stable neighborhood of
$z$ embeds $T$-equivariantly into the Zariski tangent space $T_z X$.
By the construction of $\zeta$, $(p^+)^{-1}(z)$ is thus the attracting set of
$z$ for the action of $\zeta$.  If $y \in Y$, then $\lim_{t \to 0}
\zeta(t) y \in \pi^{-1}(z)$ if and only if $y \in (p^+ \circ \pi)^{-1}(z)$.
It follows that the space $(p^+ \circ \pi)^{-1}(z)$ is a union of cells
in the decomposition. 

By assumption \ref{subsec-evenstalks}(3), the cell
decomposition is filtrable \cite{BiBi:filt}, hence the
fundamental classes of the cells give a basis for $H^*_c ( (p^+ \circ
\pi)^{-1}(z) )$, which is therefore concentrated in even degrees.

This concludes the proof of Proposition \ref{prop:evenstalks}.
\end{proof}

\begin{prop} \label{prop:evencostalks}
The cohomology of the costalk of $(\pi_* \uk_Y)^{!*}$ at any $z \in X^T$ is concentrated in even degrees.
\end{prop}

\begin{proof}
It suffices to show that the stalk of the Verdier dual $\DM((\pi_*
\uk_Y)^{!*})$ is concentrated in even degrees.  For the duration of
this proof, let us write $(-)^{!*,\chi}$ for the hyperbolic
localization with respect to $\chi$ to emphasize the dependence on
$\chi$. Then we have:
\[\DM((\pi_* \uk_Y)^{!*,\chi}) \cong \DM((\pi_*
\uk_Y)^{*!,\chi}) = (\DM(\pi_* \uk_Y))^{!*,-\chi} = (\pi_*
\uk_Y)^{!*,-\chi}[2 \dim Y],\]
where the first isomorphism is given by Theorem~\ref{thm:Braden}(ii),
the second by the definition of hyperbolic localization and the third
because $\pi$ is proper and $Y$ is smooth.  The cohomology of the
stalk at $z \in X^T$ of the right hand side is concentrated in even
degrees by Proposition~\ref{prop:evenstalks}.
\end{proof}

\subsection{Hyperbolic localization on the affine Grassmannian}
\label{subsec-hyperlocGr}

We now specialize to the case: $X$ is the affine Grassmannian $\Gr$
and $T$ is the maximal torus $\TD \subset \GD$.

Recall that the connected components of $\Gr$ are parametrized by the
group $Z(\GB)^\vee$ of characters of the center $Z(\GB) \subset \GB$.
For any $\zeta \in Z(\GB)^\vee$, let $\Gr_\zeta$ denote the
corresponding connected component.  For any $\FC \in D(\Gr)$,
we write $\FC = \oplus_{\zeta \in Z(\GB)^\vee} \FC_\zeta$ where
$\FC_\zeta$ is supported on $\Gr_\zeta$.

Fix a Levi subgroup $\LB \subset \GB$ containing the maximal torus
$\TB$.  Correspondingly, there is a Levi subgroup $\LD$ of $\GD$
containing $\TD$ whose roots are dual to those of $\LB$. Let $\chi$ be
the cocharacter of $\TD$ defined by $2\rho_\GB - 2\rho_\LB$, where
$\rho_\GB$ (resp. $\rho_\LB$) denotes the half-sum of the positive roots of
$\GB$ (resp. $\LB$).  The set of $\chi$-fixed points of $\Gr$ is
$\Gr_{\LD}$.  We denote by $R^{\GD}_{\LD}: D(\Gr) \to D(\Gr_{\LD})$
the shifted hyperbolic localization functor:
\[ R^{\GD}_{\LD}(\FC) = \bigoplus_{\zeta \in Z(\LB)^\vee}
(\FC^{*!})_\zeta [\langle \zeta, 2\rho_{\LD} - 2\rho_{\GD} \rangle].\]
As shown in~\cite[5.3.27-31]{BD}, $R^{\GD}_{\LD}$ is t-exact and
corresponds under geometric Satake to the restriction functor
$\Rep(\GB) \to \Rep(\LB)$. (This generalizes the Mirkovi\'{c}-Vilonen
weight functors which are the case when $\LD = \TD$.)

We can now use Proposition~\ref{prop:evenstalks} to prove
Theorem~\ref{thm-parityres}.

\begin{proof}[Proof of Theorem~\ref{thm-parityres}]
Recall from \cite[Theorem 4.6]{JMW2} and its proof, that every
indecomposable parity complex is a direct summand of the push forward
of the constant sheaf from a generalized Bott-Samelson resolution
$f:\Bs \to \overline{\Gr^\lambda}$.  Thus it suffices to show that $R^{\GD}_{\LD}(f_*
\uk_{\Bs})$ is a parity complex.  

Let $\TD \times \CM^*$ act on $\Bs$ and $\overline{\Gr^\lambda}$,
where $\CM^*$ acts by `loop-rotation'.  We will now show that $\TD
\times \CM^*$ and $f:\Bs \to \overline{\Gr^\lambda}$ satisfy the
assumptions of~\ref{subsec:setup} and~\ref{subsec-evenstalks} on $T$,
$X$, and $Y$.  Then by Proposition~\ref{prop:evenstalks}
(resp.~\ref{prop:evencostalks}), $((f_* \uk_{\Bs})^{*!})_\zeta $ has
$*$-even stalks (resp. costalks) at the $\TD \times \CM^*$-fixed
points (equivalently the $\TD$-fixed points) of $\Gr_{\LD}$.  But
$(f_* \uk_{\Bs})^{*!}$ is also $\LD(\OC)$-equivariant and every
$\LD(\OC)$-orbit contains a $\TD$-fixed point, thus $(f_*
\uk_{\Bs})^{*!}$ is an even complex and $R^{\GD}_{\LD}(f_* \uk_{\Bs})$
is parity.

Thus it remains to check the assumptions~\eqref{affine} of Section~\ref{subsec:setup}
and (1--3) of Section \ref{subsec-evenstalks}. We will outline why they are valid for any
Schubert variety in any partial flag variety of a Kac-Moody group (of which the varieties
$\overline{\Gr^\lambda}$ are a special case).

Assume that $\GC$ is a Kac-Moody group with maximal torus $\TC$,
Weyl group $W$ and simple reflections $S$. In
~\cite[Chapter VII]{Ku} it is shown that any Schubert 
variety in a partial flag variety for $\GC$ embeds into
the projectivization of a finite dimensional
$\TC$-representation. Also, the $\TC$-fixed points on any Schubert
variety are parametrized by an ideal in the Bruhat order on $W/W_I$,
where $W$ denotes the Weyl group and $W_I \subset W$ is a standard
parabolic subgroup. In particular, the $\TC$-fixed points on any
Schubert variety are finite. Recall (see e.g. \cite[\S 7,
Def.-Prop. 1]{GL}) that generalized Bott-Samelson resolutions may be
embedded as closed subvarieties of products of Schubert varieties for
$\GC$. It follows that ~\ref{subsec:setup}~\eqref{affine}
and~\ref{subsec-evenstalks}(1) and (3) hold for Schubert varieties and
their Bott-Samelson resolutions.

Finally, the $\TC$-fixed point corresponding to $w \in W/W_I$ in any
Schubert variety is attractive. Indeed, all weights in the tangent
space belong to the set $-w(R^+)$, where $R^+ \subset \XB(\TC)$
denotes the positive real roots. Hence if $\chi \in \YB(\TC)$ is a
cocharacter which is negative on all simple roots (which exists
because the simple roots are linearly independent in $\XB(\TC)$) then
$w \cdot \chi$ acts attractively at the fixed point corresponding to
$w$. Hence~\ref{subsec-evenstalks}(2) is satisfied for any Kac-Moody
Schubert variety.
\end{proof}

\section{Tilting modules}
\label{sec-tilting}

The aim of this section, which may be read independently of the
previous one, is to prove our main result, Theorem~\ref{thm:T=P}.

\subsection{Tilting objects in highest weight categories}
\label{subsec-hwc}


Let $\CC$ be a highest weight category with poset $\LL$ and standard (resp. costandard) objects $\Delta_\ll$ (resp. $\nabla_\ll$) for each $\ll \in \LL$.

Let $\FC(\Delta) \subset \CC$ (resp. $\FC(\nabla) \subset \CC$) denote
the full subcategory of (co)standard filtered objects, meaning $X \in
\CC$ and $X$ has a filtration whose successive quotients are (co)standard objects.  Thus $\FC(\Delta) \cap \FC(\nabla) \subset \CC$ is the full subcategory of tilting objects.

The following theorem is due to Ringel~\cite[Theorem 4 and 4*]{Ringel} and gives a useful criterion for determining if an object is tilting.

\begin{thm}
\label{thm-Ext1}
An object $X\in \FC(\Delta)$ if and only if $\Ext^1(X,\nabla_\ll)=0$  for all $\ll \in \LL$.
Dually, $X\in \FC(\nabla)$ if and only if $\Ext^1(\Delta_\ll,X)=0$  for all $\ll \in \LL$.
\end{thm}

We also mention the following result of Donkin \cite[Proposition
A4.4]{Donkin:qschur} that will be used in the appendix:

\begin{prop}
\label{prop-resol}
An object $X$ of $\CC$ is in $\FC(\nabla)$ (resp. $\FC(\Delta)$) if and only if it admits a
finite left (resp. right) resolution by tilting modules.
\end{prop}

\subsection{Parity and Tilting}
\label{subsec-paritytilt}

In this section we apply the tilting criterion, Theorem~\ref{thm-Ext1}, to parity sheaves. We
first briefly recall the setting of~\cite{JMW2}.

Let $H$ denote a connected linear complex algebraic group. Let $X$ be
a complex algebraic variety (resp. $H$-variety) together with an
algebraic stratification $X = \sqcup_{\ll \in \LL} X_\ll$ into smooth
locally closed (resp. $H$-stable) subsets.  We let $D(X)$, or
$D(X;k)$, denote the bounded (equivariant) constructible derived
category of $k$-sheaves on $X$.

Let $P(X)$ or $P(X;k)$ denote the abelian subcategory of $D(X)$
obtained as the heart of the perverse t-structure for the middle
perversity.  The objects in this category are (equivariant) perverse
sheaves and we denote the simple objects, which are parametrized by
strata $X_\ll$ and irreducible (equivariant) local systems $\LC$, by
$\ic(\ll,\LC)$, or simply $\ic(\ll)$ when $\LC= \uk_{X_\ll}$ is the
constant sheaf.

Recall that for any choice of a function $\dagger : \LL \to \ZM/2$,
which we refer to as a pariversity, there are notions of ($*$- or
$!$-) even, odd and parity complexes.  Unless stated otherwise, the
pariversity is assumed to be the trivial function $\natural: \LL \to
\ZM/2$, $\natural(\ll)=0$.  Recall that the dimension (or diamond)
pariversity $\Diamond: \LL \to \ZM/2$ is the function for which
$\Diamond(\ll)$ is given by the parity of the dimension of the stratum
$X_\ll$.  Note that if all of the strata are even dimensional, it is
equal to the $\natural$-pariversity.

We assume that for each stratum $X_\ll$ and each ($H$-equivariant)
local system $\LC$ on $X_\ll$, the (equivariant) cohomology
$H^i(X_\ll, \LC) = 0$ vanishes for all $i$ odd.

The following criterion provides a technique for showing that parity sheaves are tilting.

\begin{prop}
\label{prop-paritytilting}
Let $X$ be as above. Assume moreover that the strata of $X$ are simply
connected and that $P(X)$ is a highest weight category with highest
weight poset equal to the closure ordering on the set of strata and
whose standard (resp. costandard) objects are given by the perverse
extension sheaves,  ${}^p\JC_{!}(\lambda):={}^p j_{\lambda !}
\uk_{\lambda}[d_\lambda]$ (resp. ${}^p\JC_{*}(\lambda):={}^p j_{\lambda *}
\uk_{\lambda}[d_\lambda]$).

If a complex $\EC$ on $X$ is perverse and parity with respect to the dimension pariversity $\Diamond$, then it is tilting.
\end{prop}

\begin{proof}
By Proposition~\ref{thm-Ext1}, it suffices to demonstrate for any $\lambda\in\Lambda$ the vanishing
\[ \Ext^1(\EC,{}^p\JC_{*}(\lambda)) =0 = \Ext^1({}^p\JC_{!}(\lambda), \EC). \]
We prove the first equality.  The second follows by duality.

Consider the distinguished triangle
\[ {}^p\JC_{*}(\lambda) \to  j_{\lambda *} \uk_{\lambda}[d_\lambda] \to A \to\]
where $A =  {}^p\tau_{>0} j_{\lambda *} \uk_{\lambda}[d_\lambda] \in {}^p D^{>0}$.
By applying $\Ext^i(\EC,-)$ to this distinguished triangle in the constructible or equivariant derived category, one obtains a long exact sequence
\[ \ldots \to \Hom(\EC, A) \to \Ext^1(\EC, {}^p\JC_{*}(\lambda)) \to \Ext^1(\EC, j_{\lambda *} \uk_{\lambda}[d_\lambda]) \to  \ldots \]
The term $\Hom(\EC,A)=0$ because $\EC$ is perverse and $A \in {}^p D^{>0}$. By adjunction $\Ext^1(\EC, j_{\lambda *} \uk_{\lambda}[d_\lambda]) = \Ext^1(j_{\lambda}^* \EC, \uk_\lambda[d_\lambda])$.  By the parity assumptions and the fact that $X_\lambda$ is a single stratum, $j_{\lambda}^* \EC$ is $\Diamond$-even.  On the other hand $\uk_\lambda[d_\lambda]$ is also $\Diamond$-even, which implies that the $\Ext^1$ between them vanishes.  We have shown that the left and right terms in the sequence above vanish and therefore the middle term does too.
\end{proof}

\begin{remark}
The assumption that the strata be simply connected is made purely for
the sake of exposition.  The obvious analogue with that assumption
removed is true and proven by the same method.
\end{remark}

\subsection{A key observation}

We now restrict our attention to the affine Grassmannian for $\GD$.
Recall the notation from Section~\ref{subsec-parity}.  As mentioned
there, the affine Grassmannian satisfies the conditions needed to
define parity sheaves.  On the other hand, $P(\Gr)$ is also a highest
weight category by the geometric Satake theorem. Thus, we can use the
previous proposition in this setting.  Together with
Theorem~\ref{thm-parityconv} it gives a weak version of
Theorem~\ref{thm-tensor}.

\begin{lem}
\label{lem-obser}
Suppose $T_1$ and $T_2$ are tilting modules for $\GB$ such that the
corresponding tilting sheaves, $\TC_1$ and $\TC_2$, on $\Gr$ are
parity.  Then 
\begin{enumerate}
\item The convolution product $\TC_1 \star \TC_2$ is perverse and
  parity, and
\item The tensor product $T_1\otimes T_2$ is tilting.
\end{enumerate}
\end{lem}

\begin{proof} (1)
The convolution product $\TC_1 \star \TC_2$ is both parity by
Theorem~\ref{thm-parityconv} and perverse as $\star$ is t-exact.

(2)
The tensor product $T_1 \otimes T_2$ corresponds under geometric
Satake to the convolution product $\TC_1 \star \TC_2$.
By Proposition~\ref{prop-paritytilting} and (1), it is therefore tilting.
\end{proof}

\subsection{Reduction to simple simply-connected groups}
\label{subsec-reduction}

Our goal is to prove Theorem \ref{thm:T=P}.  We begin with the
following reduction.

\begin{lem}
\begin{enumerate}
\item \label{lem-derived}
$\GB$ satisfies $(*)$ if and only if its derived group
  $\DC(\GB)$ does.
\item \label{lem-cover}
Let $Z \subset \GB$ be a finite central subgroup and $\HB =
  \GB/Z$ be the quotient.  If $\GB$ satisfies $(*)$, then so does
  $\HB$.
\item \label{lem-product}
If $\GB = \GB_1 \times \ldots \times \GB_k$ is a product of
  connected reductive groups $\GB_i$ and each $\GB_i$ satisfies $(*)$, then $\GB$ satisfies $(*)$.
\end{enumerate}
\end{lem}

\begin{proof}
(1)
The short exact sequence $1 \to \DC(\GB) \to \GB \to
\GB/\DC(\GB) \to 1$ is Langlands dual to the short exact sequence
$1 \to Z(\GD)^0 \to \GD \to \GD/Z(\GD)^0 \to 1$.  The latter gives rise to maps
\[\Gr_{Z(\GD)^0} \to \Gr \to \Gr_{\GD/Z(\GD)^0},\]
which express $\Gr$ as a trivial cover of $\Gr_{\GD/Z(\GD)^0}$ with
fiber $\Gr_{Z(\GD)^0}$.  Thus every orbit closure in $\Gr$ is
isomorphic to an orbit closure in $\Gr_{\GD/Z(\GD)^0}$ and vice
versa.

\medskip

(2)
Let $\HD$ denote the Langlands dual of $\HB$.  The map $\GB \to \HB$
is dual to a finite covering map $\HD \to \GD$.  The latter induces an
inclusion of connected components $\Gr_{\HD} \hookrightarrow \Gr$.
Thus every orbit closure in $\Gr_{\HD}$ is isomorphic to an orbit
closure in $\Gr$.

\medskip

(3)
A dominant weight $\lambda$ of $\GB$ is a tuple $(\lambda_1, \ldots,\lambda_k)$ of dominant weights for each $\GB_i$.  The closure of $\Gr^\lambda$ in $\Gr$ is the product of the closures of $\Gr^{\lambda_i}$ in each $\Gr_{\GD_i}$.  Consider the box product $\EC = \EC(\lambda_1) \boxtimes \ldots \boxtimes \EC(\lambda_k)$ of parity sheaves.  It is parity, indecomposable, supported on the closure of $\Gr^\lambda$ and $\EC|_{\Gr^\lambda} = \uk_{\Gr^\lambda} [\dim \Gr^\lambda]$.  Thus $\EC = \EC(\lambda)$.  By assumption, the $\EC(\lambda_i)$ are each perverse and tilting.  Thus $\EC = \EC(\lambda)$ is also perverse and hence tilting.
\end{proof}

By part (\ref{lem-derived}) of the Lemma we can replace $\GB$ by its
derived subgroup $\DC(\GB)$, which is semisimple.  Any semisimple
group is a quotient of a product of simple simply connected
groups by a finite central subgroup.  Thus by parts~(\ref{lem-cover})
and~(\ref{lem-product}), it suffices to determine for each simple
simply-connected group if $(*)$ is satisfied.

\subsection{Minuscule weights and the highest short root}

We now assume that $\GB$ is simple and simply-connected.
We first check the theorem in two special cases.

\begin{lem}
\label{lem-MinuHSR}
Let $\mu$ be a minuscule highest weight and $\hsr$ denote the highest
short root of $\GB$. Then
\begin{enumerate}
\item $\EC(\mu)=\TC(\mu)=\ic(\mu)$;
\item if $p$ is a good prime for $\GB$, then $\EC(\hsr)$ is perverse and
 $\EC(\hsr)=\TC(\hsr)$.
\item if $p$ is a good prime for $\GB$ and moreover $p \nmid n +
  1$ in type $A_n$, resp. $p\nmid n$ in type $C_n$, then
  $\EC(\hsr)=\TC(\hsr) = \ic(\hsr)$.
\end{enumerate}
Although it is not necessary for what follows, we also note:
\begin{enumerate}
\item[(4)] In any characteristic, ${}^p\HC^0\EC(\hsr)$ is tilting.
\end{enumerate}
\end{lem}

\begin{proof} 
(1) The $\GD(\OC)$-orbit in $\Gr$ corresponding to the minuscule
highest weight $\mu$ is closed, thus 
$\ic(\mu) = {}^p\JC_!(\mu) = {}^p\JC_*(\mu) = \uk_\mu[d_\mu]$,
which implies $\EC(\mu) = \TC(\mu) = \uk_\mu[d_\mu]$.

\medskip

(2) 
Recall~\cite[2.3.3]{MOV} that the orbit closure $\ov\Gr^\hsr$ consists of
two strata, a point $\Gr^0$ and its complement
$\Gr^{\hsr}$, and the singularity is equivalent to 
that of the orbit closure of the minimal orbit of the corresponding
nilpotent cone of $\check{\gg}=\mathrm{Lie}(\GD)$.
Therefore we can apply the results of Section 4.3 of~\cite{JMW2}, for
the group $\GD$.

When $p$ is a good prime for $\GB$, it is in particular not one of the
primes listed in \cite[4.22(5)]{JMW2}.  It follows from Proposition
4.22 of~\cite{JMW2}, that there is a perverse parity extension of the
constant sheaf on the minimal orbit in $\check{\gg}$.  By gluing such
a parity complex with the constant sheaf on $\Gr^{\hsr}$, we obtain an
indecomposable perverse parity complex, constructible with respect to the
$\GD(\OC)$-orbits.  It is thus parity and perverse, so by Proposition~\ref{prop-paritytilting}, $\EC(\hsr)=\TC(\hsr)$.

\medskip

(3)
As a consequence of \cite[Proposition 4.23]{JMW2}, we have a short exact sequence:
\[
0 \longto i_* (k \otimes_\ZM H) \longto \p\JC_!(\hsr) \longto \ic(\hsr)
\longto 0
\]
(where $H$ is defined as the fundamental group of the root
  system consisting of the long
  roots of $\check{G}$, see \cite[Proposition 4.23]{JMW2}.)
So $\p\JC_!(\hsr) \simeq \ic(\hsr) \simeq \p\JC_*(\hsr) \simeq
\TC(\hsr)$ as
soon as $p$ does not divide $H$. Assuming that $p$ is good, we only
need to add the conditions stated for $A_n$ and $C_n$.

(4)
Let us regard $\Gr$ as a flag variety for the affine Kac-Moody group
$\GC$ associated to $\check{G}$ (see e.g. \cite[Example
4.2]{JMW2}). Let us parametrize the simple roots of $\GC$ by $\{ 0,
\dots, \ell \}$ so that $\Delta = \{ 1, \dots, \ell\}$ corresponds to the simple
roots of $\check{G}$. For any subset $I \subset \{0, \dots, \ell \}$ one
has a standard parabolic subgroup $\PC_I \subset \GC$. Let $J$
denote the subset of $\Delta$ of simple roots
which are orthogonal to the highest root of $\check{G}$ (i.e. those simple roots
corresponding to nodes which are not connected to the exceptional node
in the affine Dynkin diagram of $\check{G}$). Now consider the Bott-Samelson
space
\[
BS := \PC_{\Delta} \times_{\PC_J} \PC_{J \cup \{ \alpha_0 \}} \times_{\PC_J}
\PC_\Delta / \PC_{\Delta}.
\]
Then $BS$ is a $\PC_{J \cup \{ \alpha_0 \}} / {\PC_J} \cong
\PM^1$ bundle over $\PC_{\Delta}/{\PC_J} = \check{G}/P$ (where
$P \subset \check{G}$ denotes the standard parabolic subgroup
corresponding to $J \subset \Delta$). From this one can deduce that $\dim_{\CM} BS
= \dim \check{G}/P + 1$ and that
the $\check{T}$-fixed points on $BS$ are parametrized by $W/W_J \times
\{ id, s_0 \}$, where $W$ denotes the Weyl group of $\check{G}$ and
$s_0$ denotes the affine simple reflection. Now consider the map induced by multiplication:
\[
m : BS \to \GC/ \PC_{\Delta} = \Gr.
\]
We leave it as an exercise for the reader to show that $m$ is a
resolution of $\overline{\Gr^\hsr}$.

Let $\PC := f_* \uk_{\Bs}[\dim BS]$. 
By \cite[Lemma 4.21(1)]{JMW2} we have a short exact sequence:
\begin{equation}
\label{eqn-ses}
0 \to {}^p\JC_!(\hsr) \to {}^p\HC^0 \PC \to \p \HC^0 j_{0*} j_0^* \PC \to 0 
\end{equation}
This is a filtration by standard sheaves and, as
${}^p\HC^0\PC$ is self-dual (${}^p\HC^0$ is preserved by duality), duality gives a filtration by
costandard sheaves. Thus ${}^p\HC^0\PC$ is tilting.
\end{proof}

\subsection{Fundamental weights}

\begin{prop}
\label{prop-fund} Let $\GB$ be simple and simply connected with root system $\Phi$.  If $p > b(\Phi)$ (see Table~\ref{table:primes}), then for each fundamental weight $\varpi_i$,
\[ \EC(\varpi_i) = \TC(\varpi_i).\]
\end{prop}

\begin{proof}
We use the following method: first we express the
Weyl modules with fundamental highest weights in characteristic zero
as direct summands of tensor products of Weyl modules corresponding to
minuscule weights or the highest short root.

By part (1) of Lemma~\ref{lem-MinuHSR}, we have that $\EC(\mu) = \TC(\mu) = {}^p\JC_!(\mu)$ for any minuscule weight $\mu$, and by part (3), if $p$ is good and $\Phi$ not of type $A_n$ or $C_n$, then $\EC(\alpha_0) = \TC(\alpha_0) = {}^p\JC_!(\alpha_0)$.  Thus, in characteristic $p$, the analogous tensor product of Weyl modules corresponds under to the geometric Satake theorem to a perverse sheaf $\EC$ obtained as a convolution product of parity sheaves.  The perverse sheaf $\EC$ is therefore parity by Theorem~\ref{thm-parityconv} and tilting by Lemma~\ref{lem-obser}.  In particular, any summand of $\EC$ is perverse, parity and tilting.

On the other hand, if we know that the Weyl modules appearing as direct
summands in the tensor products in characteristic $0$ remain simple in
characteristic $p$ (and hence are indecomposable tilting modules),
then by comparing the characters we can conclude that the same 
decomposition occurs as in characteristic 0.

Thus, we need to know for which primes the Weyl modules remain simple.  This has already been done for us and the answers may be found in \cite{Jan}, \cite{Jan-H1}, \cite{Lub} or \cite{McNinch}.  The careful reader should observe that these results are logically independent of Theorems~\ref{thm-tensor} and~\ref{thm-res}.  In particular, they are obtained via Jantzen's sum formula.

In what follows, we use Bourbaki's notation \cite[Planches]{BOUR456}
for roots, simple roots, fundamental weights, etc.
For $\lambda\in\Lambda^+$, we denote by $V(\lambda)$ the Weyl module
of highest weight $\lambda$ over $\QM$.

\subsubsection{Type $A_n$}

All fundamental weights are minuscule, so there is nothing to prove.

\subsubsection{Type $B_n$}

The weight $\varpi_n$ is minuscule, and we have
\[
V(\varpi_n)^{\otimes 2} \simeq
V(2\varpi_n) \oplus V(\varpi_{n-1}) \oplus \ldots
 \oplus V(\varpi_1) \oplus V(0),
\]
all these Weyl modules being simple modulo $p$ as soon as $p > 2$
(see \cite[II.8.21]{Jan} and the references therein, particularly
\cite[Remark 3.4]{McNinch}). So we can generate all fundamental
tilting modules when $p > 2$, as claimed.

\subsubsection{Type $C_n$}

The weight $\varpi_1$ is minuscule, and for $1 \leq i \leq n$, we have
\[
\Lambda^i V(\varpi_1) \simeq V(\varpi_i) \oplus V(\varpi_{i-2}) \oplus \cdots
\]
where for convenience we set $\varpi_{0} = 0$.
All these Weyl modules remaining simple modulo $p$ as soon as
$p > n$ (again, see \cite[II.8.21]{Jan} and \cite[Remark
3.4]{McNinch}).
Moreover, the $i$-th exterior power splits as a summand of
the $i$-th tensor product for $p > i$, so $p > n$ is always
sufficient. It follows that we can generate all fundamental tilting
modules when $p > n$.

\subsubsection{Type $D_n$}

The weights $\varpi_{n-1}$ and $\varpi_n$ are minuscule, and we have
\[
\begin{array}{ccl}
V(\varpi_n)^{\otimes 2} & \simeq &
V(2\varpi_n) \oplus V(\varpi_{n-2}) \oplus V(\varpi_{n-4}) \oplus \cdots
\\
\\
V(\varpi_n)\otimes V(\varpi_{n-1}) & \simeq &
V(\varpi_n + \varpi_{n-1}) \oplus V(\varpi_{n-3}) \oplus V(\varpi_{n-5}) \oplus
\cdots
\end{array}
\]
all these Weyl modules remaining simple modulo $p$ as soon as $p > 2$
(again, see \cite[II.8.21]{Jan} and \cite[Remark 3.4]{McNinch}).
Hence we can generate all fundamental tilting modules when $p > 2$.

\subsubsection{Type $E_6$}

The minuscule weights are $\varpi_1$ and $\varpi_6$, and
the highest (short) root is $\varpi_2$. Moreover, we have
\[
\begin{array}{ccl}
\Lambda^2 V(\varpi_1) &\simeq& V(\varpi_3) \\
\Lambda^2 V(\varpi_6) &\simeq& V(\varpi_5) \\
\Lambda^2 V(\varpi_2) &\simeq& V(\varpi_4) \oplus V(\varpi_2)  
\end{array}
\]
and all these Weyl modules remain simple modulo $p$ as soon as $p > 3$
\cite{Jan-H1}. Hence we can generate all fundamental tilting modules
when $p > 3$.

\subsubsection{Type $E_7$}

The weight $\varpi_7$ is minuscule, and the highest (short) root is
$\varpi_1$. Moreover, we have
\[
\begin{array}{ccl}
V(\varpi_1)^{\otimes 2} &\simeq&
V(2\varpi_1) \oplus
V(\varpi_1) \oplus
V(\varpi_3) \oplus
V(\varpi_6) \oplus
V(0)
\\
\\
V(\varpi_6) \otimes V(\varpi_7) &\simeq&
V(\varpi_6 + \varpi_7) \oplus
V(\varpi_1 + \varpi_7) \oplus
V(\varpi_2) \oplus
\\
&&
\qquad
V(\varpi_5) \oplus
V(\varpi_7)
\\
\\
V(\varpi_5) \otimes V(\varpi_7) &\simeq&
V(\varpi_5 + \varpi_7) \oplus
V(\varpi_2 + \varpi_7) \oplus
V(\varpi_1 + \varpi_6) \oplus
\\
&&\qquad
V(\varpi_3) \oplus
V(\varpi_6) \oplus
V(\varpi_4)
\end{array}
\]
and all these Weyl modules remain simple modulo $p$ as soon as $p >
19$, because then all the weights involved lie in the fundamental
alcove.\footnote{We remark that $V(\varpi_1 + \varpi_7)$ is reducible modulo
$19$, according to \cite{Lub}.}
Thus we can generate all fundamental tilting modules when $p > 19$.

\subsubsection{Type $E_8$}

There is no minuscule weight. The highest (short) root is $\varpi_8$.
We have
\[
\begin{array}{ccl}
V(\varpi_8)^{\otimes 2} &\simeq&
V(2\varpi_8) \oplus V(\varpi_7) \oplus V(\varpi_1) \oplus V(\varpi_8)
\oplus V(0)
\\
\\
V(\varpi_7) \otimes V(\varpi_8) &\simeq&
V(\varpi_7 + \varpi_8) \oplus
V(\varpi_1 + \varpi_8) \oplus
V(2\varpi_8)  \oplus 
\\
&&
\qquad
V(\varpi_8) \oplus
V(\varpi_7) \oplus
V(\varpi_6) \oplus
V(\varpi_2) \oplus
V(\varpi_1)
\\
\\
V(\varpi_6) \otimes V(\varpi_8) &\simeq&
V(\varpi_6 + \varpi_8) \oplus
V(\varpi_7 + \varpi_8) \oplus
V(\varpi_2 + \varpi_8) \oplus
\\
&&
\qquad
V(\varpi_1 + \varpi_8) \oplus
V(\varpi_1 + \varpi_7) \oplus
V(\varpi_7) \oplus
\\
&&
\qquad
V(\varpi_6) \oplus
V(\varpi_5) \oplus
V(\varpi_3) \oplus
V(\varpi_2)
\\
\\
V(\varpi_5) \otimes V(\varpi_8) &\simeq&
V(\varpi_5 + \varpi_8) \oplus
V(\varpi_1 + \varpi_6) \oplus
V(\varpi_2 + \varpi_8) \oplus
\\
&&
\qquad
V(\varpi_6 + \varpi_8) \oplus
V(\varpi_1 + \varpi_2) \oplus
V(\varpi_3 + \varpi_8) \oplus
\\
&&
\qquad
V(\varpi_2 + \varpi_7) \oplus
V(\varpi_7) \oplus
V(\varpi_6) \oplus
\\
&&
\qquad
V(\varpi_5) \oplus
V(\varpi_4) \oplus
V(\varpi_3)
\end{array}
\]
and all these Weyl modules remain simple modulo $p$ as soon as $p >
31$, because then all the weights involved lie in the fundamental
alcove.\footnote{
We remark that $V(2\varpi_8)$ is reducible modulo $31$, according to \cite{Lub}.
} Thus we can generate all fundamental tilting
modules when $p > 31$.

\subsubsection{Type $F_4$ }

The short dominant root is $\varpi_4$, and we have
\[
\begin{array}{ccl}
\Lambda^2 V(\varpi_4) &\simeq&
 V(\varpi_1) \oplus V(\varpi_3)\\
\Lambda^3 V(\varpi_4) &\simeq&
 V(\varpi_2) \oplus V(\varpi_1 + \varpi_4) \oplus V(\varpi_3)
\end{array}
\]
and all these Weyl modules remain simple modulo $p$ as soon as
$p > 3$ \cite{Jan-H1,Lub}.

So, for $p > 3$, we can get $T(\varpi_1)$ and $T(\varpi_3)$ as direct
summands of $T(\varpi_4)^{\otimes 2}$, and $T(\varpi_2)$ as a direct
summand of $T(\varpi_4)^{\otimes 3}$.

\subsubsection{Type $G_2$}

The short dominant root is $\varpi_1$. We have
\[
\Lambda^2 V(\varpi_1) \simeq V(\varpi_1) \oplus V(\varpi_2),
\]
and these Weyl modules remain simple modulo $p$ for $p > 3$.
Thus we can generate all fundamental tilting modules when $p > 3$.

\end{proof}


\subsection{Arbitrary weights}

We can now complete the proof of our main theorem.

\begin{proof}[Proof of Theorem~\ref{thm:T=P}]
Recall that any dominant weight $\lambda \in \Lambda^+$ can be expressed as 
\[ \lambda = \sum_i a_i \varpi_i \]
for non-negative integers $a_i$.  By Proposition~\ref{prop-fund}, for every fundamental weight $\varpi_i$, $\TC(\varpi_i) = \EC(\varpi_i)$.  Thus the tensor product $\bigotimes T(\varpi_i)^{\otimes a_i}$ corresponds to a perverse sheaf that is also parity (as it is a convolution of parity sheaves) and tilting (by Lemma~\ref{lem-obser}).  The tensor product is therefore tilting and as it is of highest weight $\lambda$, contains $T(\lambda)$ as a direct summand.  We conclude that $\TC(\lambda)$ is a summand of the corresponding (parity) convolution product and thus $\TC(\lambda) = \EC(\lambda)$.
\end{proof}

\subsection{A remark on the bound in type $C_{n}$}

We do not know in all cases the exact bound on $p$ for which the
property $(*)$ holds. Recall that if Conjecture~\ref{conj-mv} is
true, then $(*)$ holds whenever $p$ is a good prime.

Stephen Donkin has pointed out to us that in type $C_n$ there exist $p$ with $2 < p \leq n$ such that not all indecomposable tilting modules can be obtained as direct summands of tensor products of the minuscule tilting module $T(\varpi_1)$ and $T(\alpha_0)$.  Thus, while it is possible that $(*)$ may hold in type
$C_n$ for all $p > 2$, a different method of proof will be required. Let us explain Donkin's argument.
 
So consider $\GB = \Sp_{2n}$, and assume $p > 2$.
The only minuscule fundamental weight is $\varpi_{1}$, and
$E := L(\varpi_{1}) = T(\varpi_{1})$ is the natural $\GB$-module, of
dimension $2n$. Now the highest short root is $\varpi_{2}$, and 
$T(\varpi_{2})$ is a direct summand of $\Lambda^2 E$ which itself is a
direct summand of  $E^{\otimes 2}$ since $p > 2$. Hence the question
is whether every indecomposable tilting module is a direct summand of
a tensor power of $E$.

Now assume $p$ divides $n$. Thus $p$ divides $\dim E$. By
\cite[Proposition 2.2]{Benson-Carlson} (which is stated for
representations of finite groups but is valid with the same proof
here), any indecomposable summand of a tensor power $E^{\otimes a}$
(with $a \geq 1$) also has dimension divisible by $p$. Hence an
affirmative answer would imply that all indecomposable tilting modules
except $T(0) = k$ have dimension divisible by $p$.

From this it would follow by induction that the dimension of each
standard module $\nabla_\lambda$, where $\lambda$ is not in the
principal block, would be divisible by $p$. Indeed, there is a short
exact sequence 
\[
0 \longto R \longto T(\lambda) \longto \nabla_\lambda \longto 0
\]
where $R$ is filtered by $\nabla_\mu$'s with $\mu < \lambda$. Hence by
induction, $p$ divides $\dim R$, and by hypothesis also $\dim
T(\lambda)$, hence also $\dim \nabla_\lambda$.  The case where $R =
0$ is the base of the induction: then $\nabla_\lambda = T(\lambda)$,
and $\lambda \neq 0$, since $\lambda$ is not in the principal block.

Now by Weyl's dimension formula, for any $m \in \NM$ we have 
\[\dim \nabla_{(m - 1)\rho_{\GB}} = m^{n^2}\]
(in general type, one gets
$m^{|\Phi^{+}|}$). Taking $m$ prime to $p$, this would force
$(m - 1)\rho_{\GB}$ to be in the principal block. In particular we could take $m = 2$ (since $p > 2$), and so $\rho_{\GB}$ itself should be in the principal block. However this is not true when $n$ is congruent to $1$ or $2$ modulo $4$, since in the basis of roots the coefficient of $\rho$ along $\alpha_n$ is $\frac 1 4 m(m + 1)$.

Since we have reached a contradiction, we cannot sharpen the bound of the Proposition in type $C_{n}$. 


\section{$q$-Characters for tilting modules}
\label{subsec-qchar}

When $p$ satisfies the conditions of the Theorem \ref{thm:T=P}
  we are able to deduce that the character of a tilting module has a
  natural graded refinement. More precisely,

\begin{cor} Suppose that $\TC(\lambda)=\EC(\lambda)$.
  Then the total cohomology of the stalk of
  the tilting sheaf $\TC(\lambda)$ at a point in $\Gr^\nu$ has the
  same dimension as the weight space $T(\lambda)^\nu$. Thus the
  dimension of the weight space has a natural graded refinement.
\end{cor}

\begin{proof}
Recall from \ref{subsec-hyperlocGr} that the weight
space functor $F_\nu$ corresponds under geometric Satake to the
summand $(\FC^{*!})_\nu[\langle \nu,2\rho \rangle ]$ of
$R^{\GD}_{\TD}$ supported at the $\TD$-fixed point $t^\nu \in \Gr^\nu$.

As explained in~\cite[Prop. 3]{Br}, the local Euler characteristic of
a sheaf $\FC$ at a torus fixed point $x$ is equal to the local Euler characteristic of
any hyperbolic localization of $\FC$. Therefore, the stalk of a
perverse sheaf in the Satake category at the point $t^\nu$
has an Euler characteristic of absolute value equal to the dimension
of the $\nu$-th weight space of the corresponding representation of
$\GD$. On the other hand, the cohomology of the stalk of the
parity sheaf $\TC(\lambda)$ is concentrated in even or odd degree,
thus its total dimension is equal to the dimension of the weight space
$T(\lambda)^\nu$.  The total cohomology of the stalk is graded and
thus the dimension of the weight space inherits a natural grading.
\end{proof}

\begin{remark} In characteristic zero (where the indecomposable tilting
  modules and simple modules coincide) the above $q$-analogue of weight
  multiplicity is due to Lusztig \cite{Lu2}. Lusztig shows that the
  $q$-characters of simple modules in characteristic zero are given by
  certain Kazhdan-Lusztig polynomials associated to the (extended)
  affine Weyl group. In fact, equipped with Lusztig's results, the
  $q$-characters of tilting modules can be deduced from the ordinary
  characters of the indecomposable tilting modules and Lustig's
  $q$-characters of simple modules in characteristic zero. Indeed, one
  can lift $\EC(\lambda)$ to a parity sheaf $\EC(\lambda, \OM)$ with
  coefficients in $\OM$, a complete local ring with residue field
  $k$. If $\KM$ denotes the fraction field of $\OM$ one has an
  isomorphism
\begin{equation} \label{p0}
\EC(\lambda, \OM) \otimes_{\OM} \KM \cong \bigoplus
\ic(\overline{\Gr^\mu}; \KM)^{\oplus m_{\mu,\lambda}}
\end{equation}
where $m_{\mu,\lambda}$ denotes the multiplicity of $\Delta_\mu$ in a
$\Delta$-flag on $T(\lambda)$. (We use that the parity sheaves with
coefficients in $\KM$ on the affine Grassmannian are the intersection
cohomology complexes and that $P(\Gr;\KM)$ is semi-simple). The
$q$-character of $\EC(\lambda)$ agrees with the $q$-character of
$\EC(\lambda, \OM)$, which in turn agrees with that of
\eqref{p0}. Hence on can deduce the $q$-character of $\EC(\lambda)$
once one knows that multiplicities $m_{\mu,\lambda}$ and the
$q$-characters in characteristic zero.
\end{remark}

\begin{remark}
In \cite{Bryl} R. Brylinski has shown that Lusztig's $q$-analogue of
weight multiplicity can be interpreted in terms of a filtration on
each weight space coming from the action of a principal nilpotent
element. It would be interesting to find a similar interpretation for
the $q$-character of tilting modules.
\end{remark}

\section{Appendix}
\label{sec-appendix}

Recall that a $\GB$-module $V$ is said to admit a good filtration if
it admits a filtration with successive quotients isomorphic to the induced
modules $\nabla_\lambda$.

In the references, the Theorems~\ref{thm-tensor} and~\ref{thm-res} are
formulated as follows.

\begin{thm}
\label{thm-tensorgood}
 If $V$ and $V'$ are $\GB$-modules admitting a good
  filtration, then so is $V \otimes V'$.
\end{thm} 

\begin{thm}
\label{thm-resgood}
 Let $\LB$ be a Levi subgroup of $\GB$. If $V$ is a
  $\GB$-module with a good filtration, then $V$ has also a good
  filtration when considered as an $\LB$-module.
\end{thm}

The aim of this appendix is to show that these formulations are
equivalent:

\begin{thm}
Theorem~\ref{thm-tensor} is equivalent to Theorem~\ref{thm-tensorgood}.
\end{thm}

\begin{thm}
Theorem~\ref{thm-res} is equivalent to Theorem~\ref{thm-resgood}.
\end{thm}

\begin{proof}
Suppose Theorems~\ref{thm-tensorgood} and~\ref{thm-resgood} are true.
If $T$ and $T'$ are tilting $\GB$-modules, then in particular they admit good
filtrations.  Thus $T \otimes T'$ also admits a good filtration, as
does $\Res^\GB_\LB T$.  On the other hand, tensor product and
$\Res_\LB^\GB$ both commute with duality.  We conclude that the
$T \otimes T'$ and $\Res^\GB_\LB T$ are tilting.

\medskip

Conversely, suppose that Theorems~\ref{thm-tensor} and~\ref{thm-res}
are true.  Let $V$ and $V'$ be $\GB$-modules that admit good
filtrations.  By Proposition~\ref{prop-resol}, they both admit finite
left resolutions by tilting modules.  The tensor product of
these resolutions is a finite left resolution of $V \otimes V'$
by tilting modules.  Applying Proposition~\ref{prop-resol} 
again, we conclude that $V \otimes V'$ admits a good filtration.

Similarly, as $\Res_\LB^\GB$ is exact and by assumption takes tilting
modules to tilting modules, $\Res_\LB^\GB$ applied to a finite left resolution
of $V$ by tilting modules is a finite left resolution of $\Res_\LB^\GB V$ by
tilting modules.  We conclude that $\Res_\LB^\GB V$ admits a good filtration.
\end{proof}

\def\cprime{$'$} \def\cprime{$'$}


\end{document}